\documentclass[a4, 12pt]{amsart}
\usepackage{amssymb,amstext,amsmath,amscd,amsthm,amsfonts,enumerate,graphicx,latexsym,color}
\usepackage[all]{xy}
\usepackage{cite}
\makeatletter
\@addtoreset{equation}{section}

\newtheorem{thm}{Theorem}[section]
\newtheorem*{thm*}{Theorem}

\newtheorem{cor}[thm]{Corollary}
\newtheorem{lem}[thm]{Lemma}
\newtheorem{prop}[thm]{Proposition}
\theoremstyle{definition}
\newtheorem{dfn}[thm]{Definition}
\newtheorem*{dfn*}{Definition}
\newtheorem{rem}[thm]{Remark}
\newtheorem{ques}[thm]{Question}

\newtheorem{ex}[thm]{Example}

\newtheorem*{nota*}{Notation}
\theoremstyle{remark}

\newtheorem*{ac}{Acknowledgments}
\newtheorem{claim}{Claim}
\newtheorem*{claim*}{Claim}
\renewcommand{\qedsymbol}{$\blacksquare$}
\numberwithin{equation}{thm}
\def\epi{\twoheadrightarrow}
\def\mon{\rightarrowtail}

\def\Hom{\mathsf{Hom}}

\def\Ext{\mathsf{Ext}}

\def\syz{\mathsf{\Omega}}

\def\cm{\mathsf{CM}}
\def\lcm{\mathsf{\underline{CM}}}
\def\db{\mathsf{D^b}}
\def\cb{\mathsf{C^b}}

\def\kb{\mathsf{K^b}}
\def\proj{\operatorname{\mathsf{proj}}}
\def\inj{\operatorname{\mathsf{inj}}}

\def\mod{\operatorname{\mathsf{mod}}}

\def\add{\operatorname{\mathsf{add}}}

\def\det{\operatorname{\mathsf{det}}}
\def\rk{\operatorname{\mathsf{rk}}}
\def\K{\mathit{K}}
\def\A{\mathcal{A}}
\def\X{\mathcal{X}}

\def\C{\mathcal{C}}

\def\T{\mathcal{T}}
\def\G{\mathcal{G}}

\def\Z{\mathbb{Z}}
\def\a{\mathsf{A}}
\def\d{\mathsf{D}}
\def\e{\mathsf{E}}

\def\H{\mathsf{H}}

\def\dim{\operatorname{\mathsf{dim}}}

\def\Ext{\operatorname{\mathsf{Ext}}}

\def\syz{\mathsf{\Omega}}

\def\dim{\operatorname{\mathsf{dim}}}

\def\lcm{\operatorname{\mathsf{\underline{MCM}}}}
\def\Hom{\operatorname{\mathsf{Hom}}}

\def\X{\mathcal{X}}
\def\A{\mathcal{A}}

\def\H{\mathsf{H}}

\def\Ext{\mathsf{Ext}}

\def\rad{\mathsf{rad}}
\def\Hom{\mathsf{Hom}}
\def\mod{\mathsf{mod}}

\def\syz{\mathsf{\Omega}}
\def\dim{\operatorname{\mathsf{dim}}}

\def\add{\operatorname{\mathsf{add}}}

\def\ker{\operatorname{\mathsf{Ker}}}
\def\cok{\operatorname{\mathsf{Coker}}}

\def\a{\mathsf{A}}
\def\d{\mathsf{D}}
\def\e{\mathsf{E}}

\def\ker{\operatorname{\mathsf{Ker}}}

\def\A{\mathcal{A}}
\def\C{\mathcal{C}}
\def\E{\mathcal{E}}
\def\G{\mathcal{G}}
\def\T{\mathcal{T}}
\def\X{\mathcal{X}}
\def\bZ{\mathbb{Z}}
\def\add{\operatorname{\mathsf{add}}}
\def\mod{\operatorname{\mathsf{mod}}}

\def\proj{\operatorname{\mathsf{proj}}}
\def\kb{\mathsf{K^b}}
\def\db{\mathsf{D^b}}

\def\Ds{\mathsf{D_{sg}}}
\def\cm{\mathsf{CM}}
\def\lcm{\underline{\mathsf{CM}}}
\def\Ext{\mathsf{Ext}}
\def\Hom{\mathsf{Hom}}
\def\syz{\mathsf{\Omega}}
\def\H{\mathsf{H}}
\def\cl{\mathsf{Cl}}
\def\lan{\langle}
\def\ran{\rangle}

\title[Classifying dense (co)resolving subcategories]{Classifying dense (co)resolving subcategories of exact categories via Grothendieck groups}
\author{Hiroki Matsui} 
\address{Graduate School of Mathematics, Nagoya University, Furocho, Chikusaku, Nagoya, Aichi 464-8602, Japan}
\email{m14037f@math.nagoya-u.ac.jp}
\date{\today}
\thanks{The author is supported by Grant-in-Aid for JSPS Fellows 16J01067.
}

\begin{document}

\begin{abstract}
Classification problems of subcategories have been deeply considered so far.
In this paper, we discuss classifying dense (co)resolving subcategories of exact categories via their Grothendieck groups.
This study is motivated by the classification of dense triangulated subcategories of triangulated categories due to Thomason.
\end{abstract}

\maketitle

\section{Introduction}
Let $\C$ be a category. Classifying subcategories means for a property $\mathbb{P}$, finding a one-to-one correspondence
$$
\xymatrix{
\{ \mbox{subcategories of } \C \mbox{ satisgying } \mathbb{P}\}  \ar@<0.5ex>[r]^-f  &
S,
\ar@<0.5ex>[l]^-g
}
$$
where the set $S$ is easier to understand. 
Classifying subcategories is an important approach to understand the category $\C$ and has been studied in various areas of mathematics, for example,  stable homotopy theory, commutative/noncommutative ring theory, algebraic geometry, and modular representation theory.

Let $\A$ be an additive category and $\X$ a full additive subcategory of $\A$.
We say that $\X$ is {\it additively closed} if it is closed under taking direct summands, and that $\X$ is {\it dense} if any object in $\A$ is a direct summand of some object of $\X$.
We can easily show that $\X$ is additively closed if and only if $\X = \add\X$ and that $\X$ is dense if and only if $\A=\add\X$. Here, $\add\X$ denotes the smallest full additive subcategory of $\A$ which is closed under taking direct summands and contains $\X$.
Therefore, for any full additive subcategory $\X$ of $\A$, $\X$ is a dense subcategory of $\add\X$ and $\add\X$ is an additively closed subcategory of $\A$.
For this reason, to classify additive subcategories, it suffices to classify additively closed ones and dense ones.
Classification of additively closed or dense subcategories has been deeply studied so far. For example, the following three kinds of subcategories have been classified by Gabriel \cite{Gab}, Hopkins and Neeman \cite{Hop, Nee}, and Thomason \cite{Thom}, respectively.
\begin{enumerate}
\item The Serre subcategories of finitely generated modules over a commutative noetherian ring.
\item The thick subcategories of perfect complexes over a commutative noetherian ring.
\item The dense triangulated subcategories of an essentially small triangulated category.
\end{enumerate}
(1) and (2) are classifications of additively closed subcategories, while (3) is a classification of dense subcategories.

Let us state the precise statement of Thomason's classification theorem.
\begin{thm}[Thomason]\label{Thomason}
Let $\T$ be an essentially small triangulated category. Then there is a one-to-one correspondence
$$
\xymatrix{
\{ \mbox{dense triangulated subcategories of } \T\}  \ar@<0.5ex>[r]^-f 
& \{\mbox{subgroups of } \K_0(\T)\}, \ar@<0.5ex>[l]^-g
}
$$
where $f$ and $g$ are given by $f(\X):=\langle [X] \mid X \in \X \rangle$ and $g(H):=\{X \in \T \mid [X] \in H\}$, respectively, and $\K_0(\T)$ stands for the Grothendieck group of $\T$.
\end{thm}

Motivated by this theorem, we discuss classifying dense (co)resolving subcategories of exact categories.
The notion of a resolving subcategory has been introduced by Auslander and Bridger \cite{AB} and that of a coresolving subcategory is its dual notion. 
(Co)resolving subcategories have been widely studied so far, for example, see \cite{APST, AR2, crspd, KS, arg, stcm}.
The main theorem of this paper is the following.
\begin{thm}[Theorem \ref{6main}]
Let $\E$ be an essentially small exact category with a (co)generator $\G$.
Then there is a one-to-one correspondence

$$
\xymatrix{
\left\{ 
\begin{matrix}
\text{dense $\G$-(co)resolving subcategories of $\E$}
\end{matrix}
\right\}  
\ar@<0.5ex>[r]^-f &
\ar@<0.5ex>[l]^-g
}
\!\!\!
\left\{ 
\begin{matrix}
\text{subgroups of $\K_0(\E)$} \\
\text{containing the image of $\G$} 
\end{matrix}
\right\},
$$
where $f$ and $g$ are given by $f(\X):=\langle [X] \mid X \in \X \rangle$ and $g(H):=\{X \in \E \mid [X] \in H\}$, respectively, and $\K_0(\E)$ stands for the Grothendieck group of $\E$.
\end{thm}
Here, the notion of a $\G$-resolving (resp. $\G$-coresolving) subcategory is a slight generalization of that of a resolving (resp. coresolving) subcategory.
Indeed, they coincide when $\G$ consists of the projective (resp. injective) objects.
The precise definitions will be given in Definition \ref{(co)res}.

The organization of  this paper is as follows.
In Section 2, we give a proof of our main theorem and several corollaries which include a correspondence between dense (co)resolving subcategories of an exact category and dense triangulated subcategories of its derived category. 
In Section 3, as applications of our results, we discuss when there are only finitely many dense (co)resolving subcategories of finitely generated modules over a left noetherian ring. 

\section{Classification of dense resolving subcategories}

In this section, we give our main result and several corollaries. 

Throughout this paper, let $\A$ be an abelian category, $\E$ an exact category, and $\T$ a triangulated category.
We always assume that all categories are essentially small, and that all subcategories are full and additive. 
For a left noetherian ring $A$, $\mod A$ denotes the category of finitely generated left $A$-modules.

We begin with recalling several notions, which are key notions of this paper.

\begin{dfn}
Let $\G$ be a family of objects of $\E$.
We call $\G$ a {\it generator} (resp. a {\it cogenerator}) of $\E$ if for any object $A \in \E$, there is a short exact sequence
$$
A' \mon G \epi A\,\,\,(\mbox{resp.}\ A \mon G \epi A')
$$
in $\E$ with $G \in \G$. 
\end{dfn}

\begin{ex}
\begin{enumerate}[\rm(1)]
\item Clearly, $\E$ is both a generator and a cogenerator of $\E$
\item If $\E$ has enough projective (resp. injective) objects, then the subcategory $\proj \E$ (resp. $\inj \E$) consisting of projective (resp. injective) objects is a generator (resp. a cogenerator) of $\E$.
\end{enumerate}
\end{ex}

Next we give the definitions of $\G$-resolving and $\G$-coresolving subcategories.

\begin{dfn}\label{(co)res}
Let $\X$ be a subcategory of $\E$ and $\G$ a family of objects of $\E$.
\begin{enumerate}[\rm(1)]
\item
We say that $\X$ is a $\G$-{\it resolving subcategory} of $\E$ if the following three conditions are satisfied.
	\begin{enumerate}[\rm(i)]
	\item 
	$\X$ is closed under extensions:  for a short exact sequence $X \mon Y \epi Z$ in $\E$, if $X$ and $Z$ are in $\X$, then so is $Y$.
	\item
	$\X$ is closed under kernels of admissible epimorphisms:  for a short exact sequence $X \mon Y \epi Z$ in $\E$, if $Y$ and $Z$ are in $\X$, then so is $X$.
	\item
	$\X$ contains $\G$.
	\end{enumerate}
If $\E$ has enough projective objects, we shall call $\X$ simply {\it resolving} if it is  $\proj \E$-resolving.
\item
We say that $\X$ is $\G$-{\it coresolving subcategory} of $\E$ if  the following three conditions are satisfied.
	\begin{enumerate}[\rm(i)]
	\item
	$\X$ is closed under extensions:  for a short exact sequence $X \mon Y \epi Z$ in $\E$, if $X$ and $Z$ are in $\X$, then so is $Y$.
	\item
	$\X$ is closed under cokernels of admissible monomorphisms:  for a short exact sequence $X \mon Y \epi Z$ in $\E$, if $X$ and $Y$ are in $\X$, then so is $Z$.
	\item
	$\X$ contains $\G$.
	\end{enumerate}
\end{enumerate}
\end{dfn}

\begin{rem}
Unlike the definition due to Auslander and Bridger \cite{AB}, we do not assume that resolving subcategories are closed under direct summands.
Therefore, our definition is rather close to the definitions in \cite{AR2, KS}.
\end{rem}

The following lemma says that dense $\G$-resolving and dense $\G$-coresolving subcategories are the same thing.

\begin{lem}\label{rescores}
Let $\X$ be a dense subcategory of $\E$.
Then $\X$ is closed under cokernels of admissible monomorphisms if and only if it is closed under kernels of admissible epimorphisms.
\end{lem}

\begin{proof}
We have only to show the `if' part. 
The `only if' part is proved by the dual argument.

Let $X \overset{f}{\mon} Y \overset{g}{\epi} Z$ be a short exact sequence in $\E$ with $X, Y \in \X$.
Since $\X$ is dense, we can take $Z' \in \E$ with $Z \oplus Z' \in \X$.
Consider a short exact sequence 
$$
X \oplus Z \overset{\left( \begin{smallmatrix} f & 0 \\ 0 & 0 \\ 0 & id_Z \end{smallmatrix} \right)}{\mon} Y \oplus Z' \oplus Z  \overset{\left( \begin{smallmatrix} g & 0 & 0 \\ 0 & id_{Z'} & 0 \\ 0 & 0 & 0 \end{smallmatrix} \right)}{\epi} Z \oplus Z'.
$$
Then $X \oplus Z$ is an object of $\X$ because $\X$ is closed under kernels of admissible epimorphisms.
From the split short exact sequence $Z \mon X \oplus Z \epi X$, we obtain $Z \in \X$ since $\X$ is closed under kernels of admissible epimorphisms.
\end{proof}

Now we recall the definitions of the Grothendieck groups of an exact category and a triangulated category.
\begin{dfn}
\begin{enumerate}[$(1)$]
\item 
Let $\E$ be an exact category. Let $F$ be the free abelian group generated by the isomorphism classes of objects of $\E$. Let $I$ be the subgroup of $F$ generated by the elements of the form $[A]-[B]+[C]$ where $A \mon B \epi C$ are short exact sequences in $\E$.
Then we define the {\it Grothendieck group} of $\E$, denoted by $\K_0(\E)$, as the quotient group $F/I$.
\item 
Let $\T$ be a triangulated category. Let $F$ be the free abelian group generated by the isomorphism classes of objects of $\T$. Let $I$ be the subgroup generated by the elements of the form $[A]-[B]+[C]$ where $A \to B \to C \to A[1]$ are exact triangles in $\T$.
Then we define the {\it Grothendieck group} of $\T$, denoted by $\K_0(\T)$, as the quotient group $F/I$.
\end{enumerate}
\end{dfn}

The following lemma is useful to calculate the Grothendieck groups.
\begin{lem}\label{vergro}
Let $\T$ be an essentially small triangulated category and $\X$ a thick subcategory of $\T$.
Then the natural homomorphism $ \varphi : \K_0(\T)/\lan [X] \mid X \in \X \ran \to \K_0(\T/ \X)$ induced by the quotient functor $\T \to \T/ \X$ is an isomorphism.
\end{lem} 
\begin{proof}
We construct the inverse $\psi: \K_0(\T/ \X) \to \K_0(\T)/\lan [X] \mid X \in \X \ran$ of $\varphi$.
\begin{claim}
Let $X, Y \in \T$. If $X$ and $Y$ are isomorphic in $\T/ \X$, then $[X]=[Y]$ in $\K_0(\T)/ \lan [X] \mid X \in \X \ran$.
\end{claim}
\begin{proof}[Proof of Claim]
Since $X$ and $Y$ are isomorphic in $\T/\X$, there is a diagram 
$$
\xymatrix{
& Z & \\
X \ar[ur]^{s} & & Y \ar[ul]_{t}}
$$
in $\T$ such that the mapping cones $C(s)$ of $s$ and $C(t)$ of $t$ are in $\X$ (see \cite[Proposition 2.1.35]{Nee1}). 
Therefore, $[X]=[Z] - [C(s)]=[Z]$ and $[Y]=[Z] - [C(t)] =[Z]$ in $\K_0(\T)/ \lan [X] \mid X \in \X \ran$. 
Thus, we have $[X]=[Y]$ in $\K_0(\T)/ \lan [X] \mid X \in \X \ran$.
\renewcommand{\qedsymbol}{$\square$}
\end{proof}
\begin{claim}
Let $X \to Y \to Z \to X[1]$ be an exact triangle in $\T/ \X$. Then $[X] - [Y] + [Z]=0$ in $\K_0(\T)/ \lan [X] \mid X \in \X \ran$.
\end{claim}
\begin{proof}[Proof of Claim]
By the definition of the Verdier quotient, any exact triangle in $\T/ \X$ is isomorphic to an exact triangle in $\T$.
Take an exact triangle $X' \to Y' \to Z' \to X'[1]$ in $\T$ which is isomorphic to $X \to Y \to Z \to X[1]$ in $\T/ \X$.
Then, using Claim 1, we have
$$
[X]-[Y]+[Z]=[X']-[Y']+[Z']=0
$$
in $\K_0(\T)/ \lan [X] \mid X \in \X \ran$.
\renewcommand{\qedsymbol}{$\square$}
\end{proof}
From Claim 2, there is a group homomorphism 
$$
\psi: \K_0(\T/ \X) \to \K_0(\T)/\lan [X] \mid X \in \X \ran, \ [X] \mapsto [X] \mbox{ mod } \lan [X] \mid X \in \X \ran.
$$
By definition, $\varphi$ and $\psi$  are inverse to each other.
\end{proof}

The following theorem is our main result of this paper.

\begin{thm}\label{6main}
Let $\E$ be an essentially small exact category with a (co)generator $\G$.
Then there are one-to-one correspondences among the following sets:
\begin{enumerate}[$(1)$]
\item
\{dense $\G$-resolving subcategories of $\E$\},
\item
\{dense $\G$-coresolving subcategories of $\E$\}, and
\item
\{subgroups of $\K_0(\E)$ containing the image of $\G$\}.
\end{enumerate}
\end{thm}

We will show this theorem only in the case that $\G$ is a generator because in the cogenerator case, it can be shown by the dual argument.
The following lemma is essential in the proof of our theorem. 
\begin{lem}\label{lem1}
Let $\G$ be a generator of $\E$ and $\X$ a dense $\G$-resolving subcategory of $\E$.
Then for an object $A$ in $\E$,  $A \in \X$ if and only if $[A] \in \langle [X] \mid X \in \X \rangle$.
\end{lem}
\begin{proof}
Define an equivalence relation $\sim$ on the isomorphism classes $\E/_{\cong}$ of objects of $\E$, as follows:
$A \sim A'$ if there are $X, X' \in \X$ such that $A \oplus X \cong A' \oplus X'$.
Set $\langle \E \rangle_{\X} :=(\E/_{\cong})/_{\sim}$ and denote by $\langle A \rangle$ the class of $A$.
Then $\langle \E \rangle_{\X}$ is an abelian group with $\langle  A \rangle + \langle B \rangle := \langle  A \oplus B \rangle$.
Indeed, obviously, $+$ is well-defined, commutative, associative, and $\lan 0 \ran$ is an identity element.
Since $\X$ is dense, for any $A \in \E$, there is $A' \in \E$ such that $A \oplus A' \in \X$, and hence $\lan A \ran + \lan A' \ran = \lan A \oplus A' \ran = \lan 0 \ran$. 
Therefore, $\lan A' \ran$ is an inverse element of $\lan A \ran$.

Let $A \overset{f}{\mon} B \overset{g}{\epi} C$ be a short exact sequence  in $\E$. Taking $A', C' \in \E$ with $A \oplus A', C \oplus C' \in \X$ and considering a short exact sequence
$$
A \oplus A' \overset{\left( \begin{smallmatrix} f & 0 \\ 0 & id_{A'} \\ 0 & 0 \end{smallmatrix} \right)}{\mon} B \oplus A' \oplus C'  \overset{\left( \begin{smallmatrix} g & 0 & 0 \\ 0 & 0 & 0 \\ 0 & 0 & id_{C'} \end{smallmatrix} \right)}{\epi} C \oplus C'.
$$
we have $B \oplus A' \oplus C' \in \X$. 
This shows $\lan B \ran - \lan A \ran - \lan C \ran = \lan B \oplus A' \oplus C' \ran =\lan 0 \ran$.
Therefore, there is a group homomorphism
$$
\varphi : \K_0(\E) \to \lan \E \ran_{\X},\ \  [A] \mapsto \lan A \ran.
$$
Note that $\lan [X] \mid X \in \X \ran$ is contained in $\ker \varphi$.

From the definition of the Grothendieck group, any element of $\K_0(\E)$ is denoted by $[A]-[B]$.
Moreover, since there is a short exact sequence
$$
B' \mon G \epi B
$$ 
in $\E$ with $G \in \G$, $[A] - [B] = [A \oplus B'] - [G]$.
Thus, any element of $\K_0(\E)$ is denoted by $[A] - [G]$ with $G \in \G$.

Let $[A] - [G]$ with $G \in \G$ be an element of $\ker \varphi$.
Since $\X$ contains $\G$, $[A] \in \ker \varphi$. This means $\lan A \ran = \lan 0 \ran$ and there are $X, X' \in \X$ such that $A \oplus X \cong X'$.
Considering the split short exact sequence
$$
A \mon A \oplus X \epi X,
$$
we obtain $A \in \X$ since $\X$ is closed under kernels of epimorphisms.
Thus, $A \in \X$ if and only if $[A] \in \lan [X] \mid X \in \X \ran$.
\end{proof}

\begin{proof}[Proof of Theorem \ref{6main}]
By Lemma \ref{rescores}, the set $(2)$ is nothing but the set $(1)$.
Therefore, we show that there is a one-to-one correspondence between the sets $(1)$ and $(3)$.

For a dense $\G$-resolving subcategory $\X$, define 
$$
f(\X) := \lan [X]  \mid X \in \X \ran,
$$
and for a subgroup $H$ of $\K_0(\E)$ containing the image of $\G$, define
$$
g(H) := \{A \in \E \mid [A] \in H\}.
$$
We show that $f$ and $g$ give mutually inverse bijections between $(1)$ and $(3)$.

First note that $g(H) :=\{A \in \E\mid [A] \in H\}$ is a dense $\G$-resolving subcategory of $\E$ for a subgroup $H$ of $\K_0(\E)$ containing the image of $\G$.
Indeed, for any object $A \in \E$, take a short exact sequence $A' \mon G \epi A$ in $\E$ with $G \in\G$.
Then $[A \oplus A'] = [A] + [A'] = [G] \in H$, and hence $A \oplus A' \in g(H)$. Thus $g(H)$ is dense.
Obviously, $g(H)$ contains $\G$.
Furthermore, for any short exact sequence $A \mon B \epi C$, the relation $[A] - [B] + [C]=0$ implies that $g(H)$ is $\G$-resolving.
Besides, $f(\X)$ is clearly a subgroup of $\K_0(\E)$ containing the image of $\G$.
As a result, $f$ and $g$ are well-defined maps between the sets $(1)$ and $(3)$

Let $H$ be a subgroup of $\K_0(\E)$ containing the image of $\G$. 
Then the inclusion $fg(H) \subset H$ is trivial.
For any $[A]-[G] \in H$ with $G \in \G$, $[A]=([A]-[G])+[G] \in H$ implies $A \in g(H)$, and thus $[A] - [G] \in fg(H)$. 
Therefore, $fg(H)=H$.

Let $\X$ be a dense resolving subcategory of $\E$ containing $\G$. 
Then the inclusion $\X \subset gf(\X)$ is trivial.
Conversely, for any $A \in gf(\X)$, since $[A] \in f(\X) = \lan [X] \mid X \in \X \ran$, we have $A \in \X$ by Lemma \ref{lem1}. 
Therefore, $gf(\X)= \X$.
Consequently, $f$ and $g$ are mutually inverse bijections between $(1)$ and $(3)$.
\end{proof}

For the rest of this section, we give several corollaries of Theorem \ref{6main}.
To state the first corollary, we have to give the definition for a given exact category to be weakly idempotent complete. 

\begin{dfn}
An exact category $\E$ is said to be {\it weakly idempotent complete} if for any morphism $p: X \to Y$ and $i: Y \to X$ with $pi=id_Y$, $p$ has a kernel and $i$ has a cokernel. 
\end{dfn}
Note that every abelian category is weakly idempotent complete.

A chain complex $N$ over $\E$ is called {\it acyclic} if for each integer $n$, the $n$-th differential $d_N^n$ factors as
$$
\xymatrix@M=5pt{
N^n \ar[rr]^{d_N^n} \ar@{->>}[dr]_{p^n} & & N^{n+1} \\
& Z^{n+1} \ar@{>->}[ur]_{i^n}
}
$$
such that $Z^{n} \overset{i^{n-1}}{\mon} N^n \overset{p^n}{\epi} Z^{n+1}$ is a short exact sequence for each integer $n$.
Then by \cite[Lemma 1.1]{Nee0}, the subcategory $\mathsf{K_{ac}^*}(\E):=\{N \in \mathsf{K}^*(\E) \mid N \mbox{ is isomorphic to an acyclic complex}\}$ of $\mathsf{K}^*(\E)$ is triangulated, where $\mathsf{K}^*(\E)$ stands for the homotopy category of $\E$ and $* \in \{\emptyset, -, +, b\}$. 
The derived category of $\E$ is by definition the Verdier quotient
$$
\mathsf{D}^*(\E) := \mathsf{K}^*(\E) / \mathsf{K_{ac}^*}(\E).
$$
For details, see \cite{Buh, Nee0}.

The following lemma is well known for the case where $\E$ is abelian.
\begin{lem}\label{groder}
Let $\E$ be a weakly idempotent complete essentially small exact category.
Then the canonical functor $\E \to \db(\E)$ induces an isomorphism $ \varphi : \K_0(\E) \to \K_0(\db(\E))$.
\end{lem}

\begin{proof}


First note that the category $\cb(\E)$ of bounded chain complexes over $\E$ is a Frobenius exact category with respect to sequences which are degreewise split short exact sequences.
We usually consider this exact structure. 
Let $X \overset{f}{\mon} Y \overset{g}{\epi} Z$ be a short exact sequence in $\E$.
Put $X \xrightarrow{i} C(id_X) \xrightarrow{p} X[1]$ to be the natural degreewise split short exact sequence in $\cb(\E)$, where $C(id_X)$ is the mapping cone of $id_X$ and $i={}^t(0, 1)$, $p=(1, 0)$.
Then we have the following commutative diagram in $\cb(\E)$:
$$
\begin{CD}
X @>i>> C(id_X) @>p>> X[1]  \\
@VfVV @VuVV @| \\
Y @>j>> C(f)  @>q>> X[1] , 
\end{CD}
$$
where $j={}^t(0, 1)$, $q=(1, 0)$, $u=
\left(
\begin{smallmatrix}
1 & 0 \\
0 & f
\end{smallmatrix}
\right)$ and each row is a short exact sequence in $\cb(\E)$.
Then by \cite[Proposition 2.12]{Buh}, the left square is a pushout diagram of $f$ and $i$ in $\cb(\E)$.
Thereby, $X \xrightarrow{f} Y \xrightarrow{j} C(f) \xrightarrow{q} X[1]$ is an exact triangle in $\kb(\E)$ by the definition of the triangulated structure of the homotopy category.
On the other hand, the mapping cone of $v=(0, g): C(f) \to Z$ is $C(v)=(\cdots \to 0 \to X \xrightarrow{f} Y \xrightarrow{g} Z \to 0 \to \cdots)$.
This complex $C(v)$ is acyclic because $X \overset{f}{\mon} Y \overset{g}{\epi} Z$ is a short exact sequence in $\E$.
Therefore, $v$ is an invertible morphism in $\db(\E)$ and we can define the morphism $q v^{-1}:Z \to X[1]$ in $\db(\E)$.
Using this morphism, we obtain the following isomorphism of triangles in $\db(\E)$:
$$
\begin{CD}
X @>f>> Y @>j>> C(f) @>q>> X[1] \\
@| @| @V\overset{v}{\cong}VV @| \\
X @>f>> Y @>g>> Z @>q v^{-1}>> X[1].
\end{CD}
$$
Thus, $X \xrightarrow{f} Y \xrightarrow{g} Z \xrightarrow{q v^{-1}} X[1]$ is also an exact triangle in $\db(\E)$. 
Hence the canonical functor $\E \to \db(\E)$ induces a homomorphism 
$$
\varphi : \K_0(\E) \to \K_0(\db(\E))\,\, [X] \mapsto [X].
$$

To show that $\varphi$ is an isomorphism, we construct the inverse $\psi: \K_0(\db(\E)) \to \K_0(\E)$ of $\varphi$.
Since $\E$ is weakly idempotent complete, the subcategory consisting of acyclic complexes forms a thick subcategory of $\kb(\E)$ (see \cite[Proposition 10.14]{Buh}), and hence the quotient functor $\kb(\E) \to \db(\E)$ induces an isomorphism
$$
\K_0(\kb(\E))/ \lan [X] \mid X \mbox{ is  acyclic } \ran \cong \K_0(\db(\E))
$$
by Lemma \ref{vergro}.

On the other hand, for an acyclic complex  $X$, it is decomposed into short exact sequences
$$
Z^i \mon X^i \epi Z^{i+1}.
$$
Therefore, we obtain equalities
$$
\sum_{n \in \mathbb{Z}} (-1)^n[X^n]=\sum_{n \in \mathbb{Z}} (-1)^n \{[Z^n]+ [Z^{n+1}] \}=\sum_{n \in \mathbb{Z}} (-1)^n[Z^n]-\sum_{n \in \mathbb{Z}} (-1)^{n+1}[Z^{n+1}]=0
$$
in $\K_0(\E)$.
Here, $Z^i = 0$ for all but finitely many integers $i$ since $X$ is bounded.
Note that for any complex $X \in \kb(\E)$, we obtain $[X]= \sum_{n \in \mathbb{Z}} (-1)^n[X^n] = \varphi(\sum_{n \in \mathbb{Z}} (-1)^n[X^n])$ in $\K_0(\kb(\E))$ by using hard truncations.
Thus $[X]=0$ for any acyclic complex $X \in \kb(\E)$, and hence the quotient functor induces an isomorphism $\K_0(\kb(\E)) \cong \K_0(\db(\E))$.

For a complex $X \in \kb(\E)$, we set $(X):=\sum_{n \in \mathbb{Z}} (-1)^n[X^n] \in \K_0(\E)$. 
If $f: X \to Y$ is an isomorphism in $\kb(\E)$, then the mapping cone $C(f)$ of $f$ is isomorphic to an acyclic complex $0$ in $\kb(\E)$.
Therefore, $C(f)$ is also an acyclic, and hence $(C(f))=0$.
Thus, we obtain $(X)=(Y)$ from a degreewise split short exact sequence $Y \to C(f) \to X[1]$.
Furthermore, an exact triangle $X \xrightarrow{f} Y \to Z \to X[1]$ in $\kb(\E)$ is isomorphic to an exact triangle of the form $X \xrightarrow{f} Y \to C(f) \to Z[1]$. Since $Y \to C(f) \to X[1]$ is a degreewise split short exact sequence, we have 
$$
(Y)-(Z)-(X) = (Y) - (C(f)) -(X)=(Y) - (C(f)) + (X[1])=0.
$$
Hence,  there is a group homomorphism $\K_0(\kb(\E)) \to \K_0(\E)$, $[X] \mapsto (X)$. Finally, we get a group homomorphism $\psi: \K_0(\db(\E)) \to \K_0(\E)$. 
Then we can easily show that $\varphi$ and $\psi$ are inverse to each other.
\end{proof} 

Combining Theorem \ref{Thomason} and Theorem \ref{6main}, we have the following corollary.

\begin{cor}\label{6cor}
Let $\E$ be a weakly idempotent complete essentially small exact category with a (co)generator $\G$.
Then there are one-to-one correspondences among the following sets:
\begin{enumerate}[$(1)$]
\item
\{dense $\G$-(co)resolving subcategories of $\E$\},
\item
\{dense triangulated subcategories of $\db(\E)$ containing $\G$\}, and
\item
\{subgroups of $\K_0(\E)$ containing the image of $\G$\}.
\end{enumerate}
\end{cor}

Let $\G$ be a class of objects in $\E$.
By Lemma \ref{rescores}, for a dense subcategory $\X$ of $\E$ containing $\G$ the following are equivalent
\begin{enumerate}
\item
$\X$ is $\G$-resolving.
\item
For any short exact sequence $A \mon B \epi C$ in $\E$, if two of $A$, $B$ and $C$ are in $\X$, then so is the third.
\end{enumerate}
Therefore, taking $\G = \proj \E$ in this corollary gives the dense version of the following theorem due to Krause and Stevenson:
\begin{thm}{\rm \cite[Theorem 1]{KraSte}}
Let $\E$ be an exact category with enough projective objects.
Then there is one-to-one correspondence between
\begin{enumerate}[$(1)$]
\item
\{thick subcategories of $\E$ containing $\proj \E$\} and
\item
\{thick triangulated subcategories of $\db(\E)$ containing $\proj \E$\}.
\end{enumerate}
\end{thm}

Next, we give a more concrete corollary.

Let $S$ be an {\it Iwanaga-Gorenstein ring} (i.e. $S$ is noetherian on both sides and $S$ is of finite injective dimension as a left $S$-module and a right $S$-module).
Let us give several remarks about Iwanaga-Gorenstein rings (cf. \cite{B, Y}).
\begin{rem}
\begin{enumerate}[$(1)$]
\item We say that a finitely generated left $S$-module $X$ is {\it maximal Cohen-Macaulay} if $\Ext_S^i(X, S)=0$ for all integers $i>0$.
 $\cm(S)$ denotes the subcategory of $\mod S$ consisting of maximal Cohen-Macaulay $S$-modules.
Then it is a Frobenius category, and hence, its stable category $\lcm(S)$ is triangulated.
\item
Natural inclusions $\cm(S) \hookrightarrow \mod S \hookrightarrow \db(\mod S)$ induce isomorphisms
$$
\K_0(\cm(S)) \cong \K_0(\mod S) \cong \K_0(\db(\mod S)).
$$
\item
Composition of the natural inclusion $\cm(S) \hookrightarrow \db(\mod S)$ and the quotient functor $\db(\mod S) \to \Ds(S):= \db(\mod S)/ \kb(\proj(\mod S))$ induces a triangle equivalence 
$$
\lcm(S) \cong \Ds(S).
$$
\end{enumerate}
\end{rem}

\begin{cor}\label{IGcor}
Let $S$ be an Iwanaga-Gorenstein ring.
Then there are one-to-one correspondences among the following sets:
\begin{enumerate}[$(1)$]
\item
\{dense $\proj(\mod S)$-(co)resolving subcategories of $\cm(S)$\},
\item
\{dense $\proj(\mod S)$-(co)resolving subcategories of $\mod S$\},
\item
\{dense triangulated subcategories of $\db(\mod S)$ containing $\proj (\mod S)$\},
\item
\{dense triangulated subcategories of $\lcm(S) \cong \Ds(S)$\}, and
\item
\{subgroups of $\K_0(\mod S)$ containing the image of $\proj(\mod S)$\}.
\end{enumerate}
\end{cor}

\begin{proof}
One-to-one correspondences among $(1)$, $(2)$, $(3)$, and $(5)$ follow from the above remarks and Corollary \ref{6cor}.

On the other hand, by Lemma \ref{vergro}, we have an isomorphism
$$
\K_0(\Ds(S)) \cong \K_0(\db(\mod S)) / \langle [P] \mid P \in \kb(\proj(\mod S)) \rangle.
$$
By using hard truncations, every object in $\kb(\proj(\mod S))$ can be obtained from finitely generated projective $S$-modules by taking mapping cones a finite number of times.
Thus, we get $ \langle [P] \mid P \in \kb(\proj(\mod S)) \rangle= \langle [P] \mid P \in \proj(\mod S) \rangle$ in $\K_0(\db(\mod S))$.
Since the natural inclusion $\mod S \hookrightarrow \db(\mod S)$ induces an isomorphism $\K_0(\mod S) \cong \K_0(\db(\mod S))$, we obtain 
$$
\K_0(\mod S) / \langle [P] \mid P \in \proj(\mod S) \rangle \cong \K_0(\db(\mod S)) / \langle [P] \mid P \in \kb(\proj(\mod S)) \rangle.
$$
Therefore, by using Theorem \ref{Thomason}, we get a bijection between $(4)$ and $(5)$.
\end{proof}

In the last two corollaries, we constructed a triangulated category $\db(\E)$ from an given exact category $\E$ and discussed their dense subcategories.   
Next, we consider the opposite direction. 
More precisely, we construct an abelian category from a given triangulated category, and then we discuss their dense subcategories.

Let us recall the definition and some basic properties of t-structures; for details, see \cite{GM}.
\begin{dfn}
A {\it t-structure} on $\T$ is a pair $(\T^{\le 0}, \T^{\ge 0})$ of subcategories in $\T$ satisfying the following conditions:
\begin{enumerate}[$(1)$]
\item
$\Hom_\T(\T^{\le -1}, \T^{\ge 0})=0$.
\item
For any object $X \in \T$, there exists an exact triangle $X' \to X \to X'' \to X'[1]$ in $\T$ with $X' \in \T^{\le -1}$ and $X'' \in \T^{\ge 0}$.
\item
$\T^{\le -1} \subset \T^{\le 0}$ and $\T^{\ge 0} \subset \T^{\ge -1}$.
\end{enumerate}
Here, $\T^{\le -n}:=\T^{\le 0}[n]$ and $\T^{\ge -n}:=\T^{\ge 0}[n]$.
We call $\T^{\le 0} \cap \T^{\ge 0}$ the {\it heart} of the t-structure   $(\T^{\le 0}, \T^{\ge 0})$.
\end{dfn}

\begin{ex}\label{6ex}
Let $\A$ be an abelian category and put
$$
\db(\A)^{\le 0} :=\{ X \in \db(\A) \mid \H^i(X)=0\ (\forall i>0)\},
$$
$$
\db(\A)^{\ge 0} :=\{ X \in \db(\A) \mid \H^i(X)=0\ (\forall i<0)\}.
$$
Then $(\db(\A)^{\le 0}, \db(\A)^{\ge 0})$ defines a t-structure on $\db(\A)$ and its heart is $\A$. We call this the standard t-structure on $\db(\A)$.
\end{ex}

In this example, the heart $\A$ is an abelian category. 
Actually, this holds in general.

\begin{prop}
Let $(\T^{\le 0}, \T^{\ge 0})$ be a t-structure on $\T$.
Then its heart $\T^{\le 0} \cap \T^{\ge 0}$ is an abelian category.
\end{prop} 

Let $(\T^{\le 0}, \T^{\ge 0})$ be a t-structure on $\T$.
Then for any object $X \in \T$, there is an exact triangle
$$
\tau_{\le -1} X \to X \to \tau_{\ge 0} X \to (\tau_{\le -1} X)[1]
$$
with $\tau_{\le -1} X \in \T^{\le -1}$ and $\tau_{\ge 0} X \in \T^{\ge 0}$.
We set $\tau_{\le n-1}:=[-n] \cdot \tau_{\le -1} \cdot [n]$, $\tau_{\ge n}:=[-n] \cdot \tau_{\ge 0} \cdot [n]$.
\begin{prop}
\begin{enumerate}[$(1)$]
\item $\tau_{\le n}$ and $\tau_{\ge n}$ define endofunctors on $\T$ for any $n \in \Z$.
\item $\tau_{\le n} \cdot \tau_{\ge n} = \tau_{\ge n} \cdot \tau_{\le n}$ for any $n \in \Z$.
\item $\H^0:= \tau_{\le 0} \cdot \tau_{\ge 0}$ is a cohomological functor from $\T$ to its heart.
Furthermore, $\H^n:=\H^0 \cdot [n] = \tau_{\le n} \cdot \tau_{\ge n}$.
\end{enumerate}
\end{prop}

For the t-structure on $\db(\A)$ defined in Example \ref{6ex}, the above definitions are nothing but the ordinary truncation functors and cohomology functors. 

Next, we introduce the notion of a bounded t-structure on a triangulated category which is a generalization of the standard t-structure on a bounded derived category.
\begin{dfn}
A t-structure $(\T^{\le 0}, \T^{\ge 0})$ on $\T$ is called {\it bounded} if $\T =\bigcup_{i, j \in \mathbb{Z}} \T^{\le i} \cap \T^{\ge j}$.
\end{dfn}
 
\begin{ex}
The standard t-structure on $\db(\A)$ is bounded. 
\end{ex}

The next proposition is a variant of Lemma \ref{groder}.
\begin{prop}
Let $(\T^{\le 0}, \T^{\ge 0})$ be a bounded t-structure on $\T$ with heart $\A$.
Then the inclusion functor induces an isomorphism $\K_0(\A) \cong \K_0(\T)$.
\end{prop}
\begin{proof}
Since any short exact sequence $A \mon B \epi C$ in $\A$ can be embedded into an exact triangle $A \to B \to C \to A[1]$ in $\T$, the inclusion functor induces a group homomorphism $\varphi: \K_0(\A) \to \K_0(\T)$.

Let $X \in \T$. Since $(\T^{\le 0}, \T^{\ge 0})$ is bounded, $\tau_{\le -n} X=0=\tau_{\ge n} X $ for $n \gg 0$, and hence $\H^n(X)=0$ for $|n| \gg 0$.
Since $\H^0$ is a cohomological functor, we can define a homomorphism $\psi: \K_0(\T) \to \K_0(\A)$ by $\psi(X) := \sum_{n \in \mathbb{Z}} (-1)^n \H^n(X)$.
Then we can easily show that $\varphi$ and $\psi$ are inverse to each other.
\end{proof}

From this proposition and Theorem \ref{6main}, we have the following corollary.
\begin{cor}
Let $\T$ be an essentially small triangulated category, $(\T^{\le 0}, \T^{\ge 0})$ a bounded t-structure on $\T$ with heart $\A$, and $\G$ a (co)generator of $\A$.
Then there are one-to-one correspondences among the following sets:
\begin{enumerate}[$(1)$]
\item
\{dense $\G$-(co)resolving subcategories of $\A$\},
\item
\{dense triangulated subcategories of $\T$ containing $\G$\}, and
\item
\{subgroups of $\K_0(\T)$ containing the image of $\G$\}.
\end{enumerate}
\end{cor}

\section{Examples}
In this section, we give some examples of module categories which have only finitely many dense resolving subcategories.

Let us start with the following remark.

\begin{rem}\label{fgrp}
Let $L$ be an abelian group.
Then there are only finitely many subgroups of $L$ if and only if $L$ is a finite group.
Indeed, `if part' is clear.
Suppose that there are only finitely many subgroups of $L$.
Then $L$ is an noetherian $\mathbb{Z}$-module and in particular, finitely generated.
Therefore, there is an isomorphism
$$
L \cong \Z^{\oplus r} \oplus  \bigoplus_{i=1}^{n} (\Z/n \Z)^{\oplus m_i},
$$
where $r, n$ and $m_i$ are non-negative integers.
We obtain $r=0$ due to our assumption as $\Z$ has infinitely many subgroups.
For this reason, $L$ is isomorphic to a finite direct sum of finite abelian groups, and thus is a finite group.
\end{rem}



From this remark and Theorem \ref{6main}, for a left noetherian ring $A$, the following two conditions are equivalent:
\begin{enumerate}
\item
There are only finitely many dense resolving subcategories of $\mod A$.
\item $\K_0(\mod A)/ \langle [P] \mid P \in \proj(\mod A) \rangle$ is a finite group.
\end{enumerate}

\subsection{The case of finite dimensional algebras}

First we consider the case of finite dimensional algebras.
Let $A$ be a basic finite dimensional algebra over a field $k$ with a complete set $\{e_1, \ldots, e_n\}$ of primitive orthogonal idempotents.
Denote $S_i:=A e_i/ \rad_A(A e_i)$ by the simple $A$-module corresponds to $e_i$.
Then by \cite[Theorem 3.5]{ASS}, $\{[S_1], \ldots, [S_n]\}$ forms a free basis of the Grothendieck group $\K_0(\mod A)$, and hence there is an isomorphism of abelian groups:
$$
\K_0(\mod A) \cong \Z^{\oplus n}. 
$$

The Cartan matrix of $A$ is an $n \times n$-matrix $C_A:=(\dim_k e_i A e_j)_{i, j=1, \ldots, n}$.
Then the above isomorphism induces the following isomorphism (see \cite[Proposition 3.8]{ASS}).
$$
\K_0(\mod A)/ \langle [P] \mid P \in \proj(\mod A) \rangle \cong \cok (\Z^{\oplus n} \xrightarrow{C_A} \Z^{\oplus n}).
$$
Therefore if $C_A$ has elementary divisors $(m_1, \cdots, m_r, 0, \cdots, 0)$, then we obtain a decomposition:
$$
\K_0(\mod A)/ \langle [P] \mid P \in \proj(\mod A) \rangle \cong \Z^{\oplus n-r} \oplus \Z/(m_1) \oplus \cdots \oplus \Z/(m_r),
$$
where $m_1, \ldots, m_r$ are not zero.
Furthermore, one has 
\[
\det C_A =
\begin{cases}
0 & (r<n) \\
m_1 \cdot m_2 \cdots m_n & (r=n).
\end{cases}
\]
As a result, the abelian group $\K_0(\mod A)/ \langle [P] \mid P \in \proj(\mod A) \rangle$ is a finite group if and only if the determinant of $C_A$ is not zero.

From this argument, we have the following corollary.

\begin{cor}
Let $A$ be a basic finite dimensional algebra over a field $k$.
Then $\mod A$ has only finitely many dense resolving subcategories if and only if its Cartan matrix has non-zero determinant. 
This is the case, the number of dense resolving subcategories is $d(m_1) \cdots d(m_n)$.
Here, $(m_1, \ldots, m_n)$ are elementary divisors of $C_A$ and $d(l)$ denotes the number of divisors of $l$.
\end{cor}

\begin{rem}
For the case of gentle algebras, Holm \cite{Hol} gives a characterization of algebras with non-zero Cartan determinant $\det C_A$.
\end{rem}

\subsection{The case of simple singularities}

Next we consider the case of simple singularities.
Let $k$ be an algebraically closed field of characteristic 0.
We say that a commutative noetherian local ring $R:= k[[x, y, z]]/(f)$ has a {\it simple (surface) singularity} if $f$ is one of the following form:
\begin{align*}
(\a_n)\quad & x^2+y^{n+1}+z^2 \ (n \ge 1),\\
(\d_n)\quad & x^2y+y^{n-1}+z^2 \ (n \ge 4),\\
(\e_6)\quad & x^3+y^4+z^2,\\
(\e_7)\quad & x^3+xy^3+z^2,\\
(\e_8)\quad & x^3+y^5+z^2.
\end{align*}
In this case, the Grothendieck group of $\mod R$ is given as follows (see \cite[Proposition 13.10]{Y}):

\begin{table}[htb]
  \begin{tabular}{|c|c|c||l|} \hline
     & $\K_0(\mod R)$ & \#\{ dense resolv. subcat. of $\mod R $\}  \\ \hline 
    $(\a_n)$ & $\bZ \oplus \bZ/(n+1)\bZ$ & $\mbox{the number of divisors of } n+1$\\ 
    $(\d_n)$ (${\it n}=\mbox{even}$) & $\bZ \oplus (\bZ/2\bZ)^{\oplus 2}$ & $5$ \\ 
    $(\d_n)$ (${\it n}=\mbox{odd}$)& $\bZ \oplus \bZ/4\bZ$ & $3$\\ 
    $(\e_6)$& $\bZ \oplus \bZ/3\bZ$ & $2$ \\ 
    $(\e_7)$ & $\bZ \oplus \bZ/2\bZ$ & $2$ \\ 
    $(\e_8)$ & $\bZ$ & $1$ \\ \hline
  \end{tabular}
\end{table}

Here, $\mathbb{Z}$ appearing in $\K_0(\mod R)$ is generated by $[R]$.
Owing to Theorem \ref{6main}, there are only finitely many dense resolving subcategories of $\mod R$.
Hence the following natural question arises.
\begin{ques}\label{6ques}
Let $R$ be a Gorenstein local ring of dimension two. 
Then does the condition $\# \{ \mbox{dense resolving subcategories of } \mod R\} < \infty$ imply that $R$ has a simple singularity?
\end{ques} 

\begin{rem}
1-dimensional simple singularities may have infinitely many dense resolving subcategories (see {\rm \cite[Proposition 13.10]{Y}}).
\end{rem}

Let $R$ be a noetherian normal local domain with residue field $k$. Denote by $\cl(R)$ the divisor class group of $R$.
Then there is a surjective homomorphism 
$$
u=\left(\begin{matrix} \rk \\ c_1 \end{matrix} \right): \K_0(\mod R) \to \mathbb{Z} \oplus \cl(R),
$$
where $\rk$ is the {\it rank function} and $c_1$ is the {\it first Charn class}. 
Moreover, $u([R])= {}^t(1, 0)$ and the kernel of $u$ is the subgroup of $\K_0(\mod R)$ generated by modules of codimension at least 2; see \cite{Bour}.
In particular, if $R$ is a 2-dimensional noetherian normal local domain with residue field $k$, we obtain a short exact sequence
$$
0 \to \lan [k] \ran \to \K_0(\mod R) \xrightarrow{\left(\begin{smallmatrix} \rk \\ c_1 \end{smallmatrix} \right)} \mathbb{Z} \oplus \cl(R) \to 0
$$
of abelian groups.
This sequence induces the following short exact sequence since $\rk(R)=1$ and $c_1(R)=0$:
$$
0 \to \langle [k], [R] \rangle/\lan [R] \ran \to \K_0(\mod R)/ \langle [R] \rangle \overset{c_1}{\to} \cl(R) \to 0.
$$
Therefore, we have an isomorphism
$$
\cl(R) \cong \K_0(\mod R)/ \langle [k], [R] \rangle
$$
and the following result is deduced from Theorem \ref{6main}:

\begin{thm}\label{classgr}
Let $R$ be a noetherian normal local domain of dimension two.
Then there is a one-to-one correspondence
$$
\left\{ 
\begin{matrix}
\text{dense resolving subcategories of $\mod R$} \\
\text{containing $k$} 
\end{matrix}
\right\}
\!\!\!
\xymatrix{
\ar@<0.5ex>[r]^-f &
\{ \mbox{subgroups of } \cl(R) \} \ar@<0.5ex>[l]^-g 
}
$$
where $f$ and $g$ are given by $f(\X):=\langle c_1(X) \mid X \in \X \rangle$ and $g(H):=\{X \in \mod R \mid c_1(X) \in H \}$ respectively.
\end{thm}

The following corollary is an answer to Question \ref{6ques}.
\begin{cor}
Let $R$ be a 2-dimensional complete Gorenstein normal local domain with algebraically closed residue field $k$ of characteristic $0$.
Then the following are equivalent:
\begin{enumerate}[$(1)$]
\item
$R$ has a simple singularity.
\item
There are only finitely many dense resolving subcategories of $\mod R$.
\item
There are only finitely many dense resolving subcategories of $\mod R$ containing $k$.
\end{enumerate}
\end{cor}
\begin{proof}
$(1) \Rightarrow (2)$: If $R$ has a simple singularity, then $\K_0(\mod R)/ \langle [R] \rangle$ is a finite group; see \cite[Proposition 13.10]{Y}.
Thus, Theorem \ref{6main} shows that $\mod R$ has only finitely many dense resolving subcategories.

$(2) \Rightarrow (3)$: This implication is trivial.

$(3) \Rightarrow (1)$: From Theorem \ref{classgr}, $\cl(R)$ has finitely many subgroups. Therefore, $\cl(R)$ is a finite group, and thus by \cite[Corollary 3.3]{DITV} we have $\syz \cm(R) = \add G$ for some module $G$, where $\syz \cm(R)$ stands for the category of first syzygies of maximal Cohen-Macaulay $R$-modules.
Now, since $R$ is Gorenstein, $\cm(R) = \syz \cm(R) $ has only finitely many indecomposable objects up to isomorphism.
Consequently, $R$ has a simple singularity from \cite[Theorem 8.10]{Y}.
\end{proof}

\begin{ex}
Let $R$ be a 2-dimensional simple singularity of type $(\a_1)$. Namely, $R= k[[x, y, z]]/(x^2 + y^2 + z^2)$. Then the indecomposable maximal Cohen-Macaulay $R$-modules are $R$ and the ideal $I=(x+ \sqrt{-1}y, z)$ up to isomorphism. 
Then the dense resolving subcategories of $\mod R$ are:
\begin{itemize}
\item
$\mod R$, and
\item
$\{M \in \mod R \mid \syz^2 M \cong R^{\oplus n} \oplus I^{\oplus 2m} \mbox{ for some } m, n \in \Z_{\ge 0} \}$.
\end{itemize}
\end{ex}

\begin{proof}
By Corollary \ref{IGcor}, there is a one-to-one correspondence between the set of dense resolving subcategories of $\mod R$ and the set of dense resolving subcategories of $\cm(R)$.
This correspondence assigns a dense resolving subcategory $\X$ of $\cm(R)$ to a dense resolving subcategory $\{M \in \mod R \mid \syz^2 M \in \X \}$ of $\mod R$.
Therefore, we have only to check that the dense resolving subcategories of $\cm(R)$ are $\cm(R)$ and $\{X \in \cm(R) \mid X \cong R^{\oplus n} \oplus I^{\oplus 2m} \mbox{ for some } m, n \in \Z_{\ge 0} \}$.
Note that $\Ext_R^1(I, I) \cong \Hom_R(\underline{\Hom}_R(I, I), E_R(k)) \cong k$ by Auslander-Reiten duality, see \cite[Lemma 3.10]{Y}.
Therefore, every non-split short exact sequence starting and ending at $I$ is isomorphic to the short exact sequence $I \mon R^2 \epi I$, see \cite[Chapter 9]{Y}.

The following claim gives us the complete information on extensions in $\cm(R)$.
\begin{claim*}
Let $\sigma: I^{\oplus n} \mon W \epi I^{\oplus m}$ be a short exact sequence in $\cm(R)$.
Then $W \cong R^{\oplus 2i} \oplus I^{\oplus n+m-2i}$ for some $0 \le i \le \min\{n, m\}$.
\end{claim*}
\begin{proof}[Proof of Claim]
If $\sigma$ splits, then this statement is clear.
Therefore we may assume that $\sigma \neq 0$ (i.e., $\sigma$ does not split).
We show the statement by induction on $N:=\min\{n, m\}$.

First consider the case where $N=1$.
We only show the claim when $n=1$ because the proof for $m=1$ is similar.
Since $0 \neq \sigma \in \Ext_R^1(I^{\oplus m}, I) \cong \Ext_R^1(I, I)^{\oplus m}$, there is a canonical inclusion $I \mon I^{\oplus m}$ such that the pullback of $\sigma$ by this inclusion is not zero. 
To be precise, there is a commutative diagram 
$$
\xymatrix@M=8pt{ 
& I^{\oplus m-1} \ar@{=}[r] & I^{\oplus m-1} \\
\sigma: I \ar@{>->}[r] & W \ar@{->>}[r] \ar@{->>}[u] \ar@{}[rd]|{\text{PB}} & I^{\oplus m} \ar@{->>}[u] \\
\sigma': I \ar@{>->}[r] \ar@{=}@<-1.5ex>[u] & W' \ar@{->>}[r] \ar@{>->}[u] & I \ar@{>->}[u]
}
$$ 
such that each of the rows and columns is a short exact sequence in $\cm(R)$, and moreover $\sigma'$ does not split.
Therefore, $\sigma':I \mon W' \epi I$ is isomorphic to $I \mon R^{\oplus 2} \epi I$.
In particular, $W'$ is isomorphic to $R^{\oplus 2}$. 
On the other hand, since $\Ext_R^1(I, R)=0$, the short exact sequence $W' \mon W \epi I^{\oplus m-1}$ is a split short exact sequence.
Thus, $W \cong R^{\oplus 2} \oplus I^{\oplus m-1}$.

Next, consider the case where $N > 1$.
Suppose $N=m$. 
The case $N=n$ is similarly handled.
Since $0 \neq \sigma \in \Ext_R^1(I^{\oplus m}, I^{\oplus n}) \cong \Ext_R^1(I, I^{\oplus n})^{\oplus m}$, there is a canonical inclusion $I \mon I^{\oplus m}$ such that the pullback of $\sigma$ by this inclusion is not zero. 
Therefore, there is a commutative diagram
$$
\xymatrix@M=8pt{ 
& I^{\oplus m-1} \ar@{=}[r] & I^{\oplus m-1} \\
\sigma: I^{\oplus n} \ar@{>->}[r] & W \ar@{->>}[r] \ar@{->>}[u] \ar@{}[rd]|{\text{PB}} & I^{\oplus m} \ar@{->>}[u] \\
\sigma': I^{\oplus n} \ar@{>->}[r] \ar@{=}@<-1.5ex>[u] & W' \ar@{->>}[r] \ar@{>->}[u] & I \ar@{>->}[u]
}
$$ 
such that each of the rows and columns is a short exact sequence in $\cm(R)$, and furthermore $\sigma'$ does not split.
By using the same argument as above, $W'$ is isomorphic to $R^{\oplus 2} \oplus I^{\oplus n-1}$ because $\sigma'$ does not split.
Since $\Ext_R^1(I, R)=0$, we obtain the following isomorphism of short exact sequences
$$
(W' \mon W \epi I^{\oplus m-1}) \cong (R^{\oplus 2} \mon R^{\oplus 2} \epi 0) \oplus (I^{\oplus n-1} \mon W'' \epi I^{\oplus m-1}).
$$
From the induction hypothesis, $W''$ is isomorphic to $ R^{\oplus 2i} \oplus I^{\oplus n+m-2i-2}$ for some $0 \le i \le \min\{n-1, m-1\}= N-1$.
Thus, $W \cong R^{2i+2} \oplus I^{\oplus n+m-2i-2}$.
Now we are done.
\renewcommand{\qedsymbol}{$\square$}
\end{proof}

From the claim, $\X:= \{X \in \cm(R) \mid X \cong R^{\oplus n} \oplus I^{\oplus 2m} \mbox{ for some } m, n \in \Z_{\ge 0} \}$ forms a dense resolving subcategory of $\cm(R)$.
Indeed, $\X$ is dense because every maximal Cohen-Macaulay $R$-module is isomorphic to a finite direct sum of indecomposable ones, and closed under extensions by the claim.
To show that $\X$ is closed under kernels of admissible epimorphisms, consider a short exact sequence: 
$$
R^{\oplus n_1} \oplus I^{\oplus m_1} \mon R^{\oplus n_2} \oplus I^{\oplus 2 m_2} \epi R^{\oplus n_3} \oplus I^{\oplus 2 m_3}.
$$
Since $\Ext_R^1(I, R)=0$, this short exact sequence is isomorphic to 
$$
(I^{\oplus m_1} \mon R^{\oplus r} \oplus I^{\oplus 2 m_2} \epi I^{\oplus 2 m_3}) \oplus (R^{\oplus n_1} \mon R^{\oplus n_1 + n_3} \epi R^{\oplus n_3}).
$$
Then, by the claim, $R^{\oplus r} \oplus I^{\oplus 2 m_2}$ must be isomorphic to
$$
R^{\oplus 2i} \oplus I^{\oplus m_1+2m_3-2i}
$$
for some integer $i$.
Therefore, we obtain an equality $2 m_2 = m_1+2m_3-2i$ as $\cm(R)$ is a Krull-Schmidt category, see \cite[Proposition 1.18]{Y}.
Thus, $R^{\oplus n_1} \oplus I^{\oplus m_1} =R^{\oplus n_1} \oplus I^{\oplus 2(m_2 - m_3 -i)} \in \X$.
This shows that $\X$ is closed under kernels of admissible epimorphisms. 
Consequently, $\X$ is a dense resolving subcategory of $\cm(R)$.

As we have already discussed before, simple singularity of type $(\a_1)$ has only two dense resolving subcategory; $\K_0(\mod R)/\lan [R] \ran \cong \Z/2\Z$ has only two subgroups. 
Hence $\X \neq \cm(R)$ is a unique non-trivial dense resolving subcategory of $\cm(R)$.
\end{proof}

\begin{ac}
The author is grateful to his supervisor Ryo Takahashi for his many helpful comments.  
\end{ac}

\end{document}